\documentclass[12pt,twoside]{amsart}
\usepackage{amssymb}
\usepackage{amscd}
\usepackage{xy}

\xyoption{arrow}
\xyoption{matrix}
\xyoption{ps}
\xyoption{dvips}

\nonstopmode

\numberwithin{equation}{section}
\font\tengothic=eufm10 scaled\magstep 1
\font\sevengothic=eufm7 scaled\magstep 1
\newfam\gothicfam
          \textfont\gothicfam=\tengothic
          \scriptfont\gothicfam=\sevengothic

\newcommand{\fc}{\mathfrak c}
\newcommand {\PP}{\mathbb{P}}
\newcommand{\cI}{\mathcal{I}}

\newcommand{\cL}{\mathcal{L}}

\newcommand{\ffi}{\varphi}

\DeclareMathOperator{\pnt}{\raise 0.5mm \hbox{\large\bf.}}

\newenvironment{caselist}{\begin{list}{ }{\labelwidth1cm
  \leftmargin1.7cm \labelsep0.5cm \rightmargin0cm}}{\end{list}}

\newtheorem{theorem}{Theorem}[section]
\newtheorem{lemma}[theorem]{Lemma}
\newtheorem{proposition}[theorem]{Proposition}
\newtheorem{corollary}[theorem]{Corollary}

\theoremstyle{definition}
\newtheorem{definition}[theorem]{Definition} 
\newtheorem{remark}[theorem]{Remark}
\newtheorem{example}[theorem]{Example}

\newtheorem{notation}[theorem]{Notation}
\newtheorem{question}[theorem]{Question}

\begin{document}
\title{Minimal links and a result of Gaeta}

\author[Juan Migliore]{Juan Migliore${}^*$}
\address{
Department of Mathematics, University of Notre Dame, Notre Dame, IN
46556, USA}
\email{Juan.C.Migliore.1@nd.edu}
\author[Uwe Nagel]{Uwe Nagel${}^{+}$}
\address{Department of Mathematics,
University of Kentucky, 715 Patterson Office Tower,
Lexington, KY 40506-0027, USA}
\email{uwenagel@ms.uky.edu}

\thanks{${}^*$ The work for this paper was done while the first
author was sponsored by the National Security Agency under Grant
Number H98230-07-1-0036.\\
${}^+$ The work for this paper was done while the second
author was sponsored by the National Security Agency under Grant
Number H98230-07-1-0065.\\
}

\begin{abstract}
If $V$ is an equidimensional codimension $c$ subscheme of an $n$-dimensional projective space, and $V$ is linked to $V'$ by a complete intersection $X$, then we say that $V$ is {\em minimally linked} to $V'$ if $X$ is a codimension $c$ complete intersection of smallest degree containing $V$.  Gaeta showed that if $V$ is any arithmetically Cohen-Macaulay (ACM) subscheme of codimension two then there is a finite sequence of minimal links beginning with $V$ and arriving at a complete intersection.  Gaeta's result leads to two natural questions.

First, in the codimension two, non-ACM case, there is no hope of linking $V$ to a complete intersection.  Nevertheless, an analogous question can be posed by replacing the target ``complete intersection'' with ``minimal element of the even liaison class" and asking if the corresponding statement is true.  Despite a (deceptively) suggestive recent result of Hartshorne, who generalized a theorem of  Strano, we give a negative answer to this question with a class of counterexamples  for codimension two subschemes of projective $n$-space.

On the other hand, we show that there {\em are} even liaison classes of non-ACM curves in projective 3-space for which every element admits a sequence of minimal links leading to a minimal element of the even liaison class.  (In fact, in the classes in question, even and odd liaison coincide.)

The second natural question arising from Gaeta's theorem concerns higher codimension.  In earlier work with Huneke and Ulrich, we show that the statement of Gaeta's theorem as quoted above is false if ``codimension two" is replaced by ``codimension $\geq 3$," at least for subschemes that  admit a sequence of links to a complete intersection (i.e. {\em licci} subschemes).  Here we show that in the non-ACM situation, the analogous statement is also false.

However, one can refine the question for codimension 3 licci subschemes by asking if it is true for {\em arithmetically Gorenstein, codimension 3} subschemes, which Watanabe showed to be licci.  Watanabe's work was extended by Hartshorne, who showed that  the {\em general} such subscheme of fixed Hilbert function and of dimension 1 can be obtained by a sequence of strictly ascending biliaisons from a linear complete intersection.  (Hartshorne, Sabadini and Schlesinger proved the analogous result for arithmetically Gorenstein zero-dimensional schemes.)  In contrast to the previous results, here we show that in fact for any codimension 3 arithmetically Gorenstein subscheme, in any projective space, a sequence of minimal links {\em does} lead to a complete intersection, giving a different extension of Watanabe's result.  Furthermore, we extend Hartshorne's result by removing the generality assumption as well as the dimension assumption.
\end{abstract}

\maketitle


\section{Introduction}

The study of liaison (or linkage) has a long history.  The first important results were obtained by Gaeta, and subsequently the next flurry of activity occurred in the 70's and early 80's, and then in the last decade there has been possibly the biggest surge of interest.  We refer to \cite{migbook} and to \cite{MN6} for a treatment of the basics of liaison theory, including an overview of what was known until roughly 2000.  An important point, though, is that it is in many ways more natural to consider {\em even} liaison classes, i.e. the equivalent relation generated by linking an even number of times.  (See below for details.)  The theory has also branched out in the directions of Gorenstein liaison and complete intersection liaison.  This paper deals exclusively with complete intersection liaison, and from now on unless otherwise specified, any link is assumed to be by a complete intersection.

In any even liaison class, a very important question has been to find distinguished elements of the class.  This brings up two questions: what does ``distinguished'' mean, and how do you find such elements? We begin with the former question.

  In certain even liaison classes, it is possible to arrive, after a finite number of links, to a complete intersection.  These are the so-called {\em licci} classes, and here we can view the complete intersections as being the most distinguished ones.  An important question, still open in codimension $\geq 3$, is to determine precisely which subschemes of projective space are licci.  In codimension two, is was shown by Gaeta \cite{gaeta-aCM} that the licci subschemes are exactly the arithmetically Cohen-Macaulay (ACM) subschemes.  It is his method of proof, described below, which inspired the work described in this paper.

For codimension two subschemes of projective space, it was shown in \cite{MDP}, \cite{BBM}, \cite{BM6}, \cite{uwepaper}, \cite{nollet} in different settings that even liaison classes have a very precise structure, called the {\em Lazarsfeld-Rao property}.  At the heart of this property is the notion of the {\em minimal} elements of the even liaison class.  These subschemes simultaneously (!) are minimal with respect to degree, arithmetic genus, and shift of the intermediate cohomology  modules (see Definition \ref{def of inter cohom}).  Furthermore, the entire even liaison class can be built up from the minimal elements in a prescribed way.  There are actually two nearly equivalent ways to achieve this, but the one that we will mention here is that of {\em elementary biliaisons}.  In our context, we will say that codimension two subschemes $V$ and $V'$ are related by an elementary biliaison if there is a hypersurface, $S$, containing both $V$ and $V'$, such that $V \sim V' + nH$ on $S$ for some $n \in \mathbb Z$, where $\sim$ refers to linear equivalence and $H$ is the class of a hyperplane section on $S$.  The structure of the even liaison classes of  curves in $\mathbb P^3$, using the elementary biliaison point of view, was achieved by Martin-Deschamps and Perrin \cite{MDP}.  It was extended by Hartshorne \cite{hartshorne-toulouse} to the case of codimension two in an arbitrary projective space.  For curves in $\mathbb P^3$ it was extended further by Strano, who characterized the minimal curves in an even liaison class as being precisely those that do not admit an elementary descending biliaison class (i.e. $V'$ is minimal above if $n$ cannot be negative).  This result was in turn extended again by Hartshorne (cf.\ Theorem \ref{hart-thm} below).

In higher codimension, liaison theory forks giving {\em Gorenstein liaison} and {\em complete intersection (CI) liaison}.  In the former case, it is no longer true that the notions of minimality mentioned above coincide  -- see \cite{migbook} for a discussion.  However, it is an open question for CI liaison.  In any case, in this paper we will use the shift of the intermediate cohomology modules as our measure.  See section \ref{liaison facts} for details.

Having described the set of distinguished elements of an even liaison class, we turn to the question of how one finds such elements.  In the case of codimension two ACM subschemes of projective space, this was also solved by Gaeta in the process of showing that the ACM subschemes are licci.   We now describe, in a more general setting, the construction that he (essentially) used for his result.  Let $V$ be a codimension $c$ subscheme of projective space and let $I_V$ be its homogeneous ideal.  It is clear that there is a well-defined integer $a_1$ which is the least degree in which $I_V$ is non-zero.  More generally, there is a non-decreasing sequence of integers $a_1,a_2,\dots,a_c$, where $a_i$ is the least degree in which there is a regular sequence in $I_V$ of length $i$.   We say that a complete intersection $X$ is a {\em complete intersection of least degree containing $V$} if $I_X \subset I_V$ and the generators of $I_X$ have degrees $a_1,a_2,\dots,a_c$.  If $V$ is linked by $X$ to a subscheme $V'$, we say that $V$ is {\em minimally linked} to $V'$.  Notice that this is not symmetric in general, since $V'$ may admit a regular sequence with generators of smaller degree.  Notice also that this extends in a natural way to the case of artinian ideals.

In the codimension two case, Gaeta proved that any ACM subscheme is minimally linked in a finite number of steps to a complete intersection.  In higher codimension it is known that not all ACM subschemes are in the {\bf li}nkage {\bf c}lass of a {\bf c}omplete {\bf i}ntersection (or {\em licci}).  However, for those that are licci, it is natural to ask if there exists a sequence of minimal links arriving finally at a complete intersection.  It was shown by Huneke, Ulrich and the current authors \cite{HMNU} that this is {\em not} the case: in any codimension $\geq 3$, there exist subschemes which are licci, but cannot be minimally linked to a complete intersection.  This paper builds on that work.

Gaeta also began the study of the non-ACM case by considering the curves $C \subset \mathbb P^3$ that have the property that a sequence of two minimal links can be found that returns to $C$ (i.e. the first residual does not admit a regular sequence with generators of smaller degree).  He referred to such a curve as being {\em ridotta seconda di se stessa}.  The first author also studied such curves, using the term ``locally minimal."  He remarked (\cite{J-thesis}, page 130) that ``it is natural to ask when locally minimal curves are actually minimal."  The same question can very naturally be posed in codimension two in any projective space, and in a licci class in any codimension (but here instead of minimal elements, we simply ask whether the licci subscheme can be minimally linked to a complete intersection).  It can also be posed in even greater generality, using the shift of the intermediate cohomology modules as the measure of minimality.

In view of the result for licci subschemes in  \cite{HMNU}  on the one
hand, and Gaeta's result described above on the other hand, there
are two natural questions that one can immediately pose, both of which are
answered in this paper.  First, what about the 
non-ACM case (both codimesion two and higher codimension)?  Second, what about arithmetically Gorenstein
subschemes?  We now expand on these questions separately.

 It is very natural to ask whether at least in codimension two, the process of linking by complete intersections of least degree can always be used to arrive at a minimal element in the even liaison class.  This would be nice, for example, because it would give a natural and very elementary algorithm to obtain a minimal element in any such even liaison class.  (More complicated algorithms do exist, e.g. \cite{MDP}, \cite{GLM}.)  A very suggestive result in this direction is the following theorem of Hartshorne, which generalizes Strano's result mentioned above:

\begin{theorem}[\cite{hartshorne-toulouse}] \label{hart-thm}
If $V$ is a codimension 2 subscheme whose degree is not minimal in its  even liaison class, then $V$ admits a strictly decreasing elementary biliaison.
\end{theorem}

The amazing (to us) fact is that this does {\em not} imply that $V$ can be minimally linked (in an even number of steps) to a minimal element.  This is  one of the two main areas of focus of this paper.  First, we use a broad range of methods from liaison theory and minimal free resolutions  to give a class of examples of codimension $c$ subschemes in any $\mathbb P^n$ ($n \geq 3$, $c \geq 2$) that are not minimally linked in any number of steps to a minimal element of the even liaison class.  More precisely, we have

\begin{theorem}
  \label{thm-non-min-link}
Given any two integers $c \geq 2, d \geq 0$, there is a  locally
Cohen-Macaulay subscheme $X \subset \PP^{c+d}$ of dimension $d$ that
is not a subscheme of any hyperplane of $\PP^{c+d}$, such that $X$ is not minimal in its even liaison class, but $X$ cannot be linked to any minimal element in its even liaison
class by way of {\em minimal} links.  When $c = 2$, $X$ is in fact
linked in two steps (but not minimally linked) to a minimal element of its even liaison class.

Furthermore, if $d \geq 0$ then $X$ can be chosen as a licci subscheme, and if $d \geq 1$ it also can be chosen as a non-ACM subscheme.

\end{theorem}

If $c =2 $, then $X$ must be necessarily  non-ACM by Gaeta's
theorem. The codimension two part of this result is given in Theorem
\ref{thm-lcm-non-min-licci}. Since it was shown in \cite{HMNU} that
in codimension $\geq 3$ there are licci subschemes that are not {\em
minimally} licci, it remains to produce non-ACM subschemes that
cannot be minimally linked to a minimal element of their even
liaison class.  This is done in Proposition~\ref{hypersurf}.  The
idea is to apply hypersurface sections of large degree to suitable
subschemes with the same property (e.g. those obtained in Section
\ref{codim 2}), as was done in \cite{HMNU}.  However, it is somewhat
more delicate since the minimal generators do not behave as well
under hypersurface sections as they do in the ACM case. The reason that we do not know that $X$ is linked in two steps to a minimal element, in codimension $\geq 3$, is that the hypersurface section of a minimal element is not necessarily minimal.  Note also that taking cones  over curves in $\PP^3$ (where the constructions are simpler, as described in Theorem \ref{no min link to min})
is not sufficient to establish the above theorem in codimension two as 
the resulting
subschemes are not locally Cohen-Macaulay.

Theorem \ref{thm-non-min-link} naturally leads to the  question of whether the ACM class is the only liaison class of curves in $\mathbb P^3$ with the property that minimal links always lead to a minimal element.  We show that this is not the case, by proving that the same property holds in the liaison class  corresponding to a intermediate cohomology module that is $n$-dimensional (over the field $k$) and concentrated in one degree.  The easiest example of such a liaison class is that of two skew lines (where $n=1$).

We now turn to the arithmetically Gorenstein subschemes of codimension three.  In many ways the known results about codimension three arithmetically Gorenstein subschemes closely parallel corresponding facts about codimension two ACM subschemes (see for instance \cite{migbook}, \cite{denegri-valla}, \cite{HTV}, \cite{GM5}).  Furthermore, it is a result of Watanabe \cite{wat} that all arithmetically Gorenstein subschemes are licci.  On the other hand, from \cite{HMNU} we know that codimension three licci subschemes are not necessarily minimally licci.  So the codimension three arithmetically Gorenstein subschemes hang in the balance.
In this paper we show that they are in fact minimally licci.  The argument goes back to the structure theorem of Buchsbaum and Eisenbud, and a careful study of Watanabe's construction.

Watanabe's result has been extended for curves in $\mathbb P^4$ in \cite{harts-Coll} (resp.\ points in $\mathbb P^3$ in \cite{HSS}) by showing that  a {\em general} arithmetically Gorenstein curve (resp. set of points) with given Hilbert function can be obtained from a line (resp. single point) by a series of ascending complete intersection biliaisons.  Since the result is obvious for complete intersections, the heart of the proof lies in showing that the general arithmetically Gorenstein curve in $\mathbb P^4$ (resp.\ set of points in $\mathbb P^3$) with given Hilbert function admits a series of strict descending complete intersection biliaisons down to a complete intersection.  We also extend this result in Theorem \ref{thm-Gorenst} in two ways: we remove the generality condition, and we allow arbitrary $\mathbb P^n$.  As noted above, results about strictly descending complete intersection biliaisons do not necessarily imply results about minimal links, but in this case both results do hold.

Gaeta in fact proved more for arithmetically Cohen-Macaulay codimension two subschemes.  That is, if one makes links at each step which are not necessarily minimal but still use only minimal generators, one still obtains a complete intersection in the same number of steps as occurs using minimal links.  We show in Example \ref{ex-counter} that this does not extend to the codimension three Gorenstein situation.

It is a satisfactory outcome that we were able to show that in some cases the minimal linkage property does hold, while in others it does not.
\medskip

{\bf Acknowledgements.}  Our debt to Craig Huneke and Bernd Ulrich, our co-authors of \cite{HMNU}, is obvious and deep, and we extend our thanks to them.  In the course of writing this paper, many experiments were carried out using the computer algebra system CoCoA \cite{CoCoA}.  Finally, we are very grateful to Enrique Arrondo and Rosa Mar\'\i a Mir\'o-Roig for encouraging us to write a paper related to Gaeta's work.

\tableofcontents


\section{Some facts from liaison theory} \label{liaison facts}

Let $R = k[x_0,\dots,x_n]$, where $k$ is a field.   We refer to \cite{migbook} for details of liaison theory, and here present just the basic facts that we will need.

\begin{notation} If $\mathcal F$ is a sheaf on projective space $\mathbb P^n$, we denote by $H^i ({\mathcal F})$ the cohomology group $H^i(\mathbb P^n, {\mathcal F})$ and by $h^i({\mathcal F})$ its dimension as a $k$-vector space.  We denote by $H^i_*({\mathcal F})$ the graded $R$-module $\bigoplus_{t \in {\mathbb Z}} H^i({\mathcal F}(t))$.
\end{notation}

\begin{definition}\label{def of inter cohom} If $V \subset \mathbb P^n$ is a closed subscheme, we define its {\em intermediate cohomology modules} (sometimes also called {\em Hartshorne-Rao modules} or {\em deficiency modules}) to be the graded modules $H^i_*({\mathcal I}_V)$, for $1 \leq i \leq \dim V$.
\end{definition}

If $V_1$ is an equidimensional subscheme of $\mathbb P^n$ and $X$ is a complete intersection properly containing $V_1$ then $X$ links $V_1$ to a {\em residual} subscheme $V_2 \subset \mathbb P^n$ defined by the saturated ideal $I_{V_2} = I_X : I_{V_2}$.  The residual $V_2$ is also equidimensional, and because of the assumption that $V_1$ is equidimensional we have $I_X : I_{V_2} = I_{V_1}$ as well.   The dimensions of $V_1$, $V_2$ and $X$ all are equal.  We write $V_1 \stackrel{X}{\sim} V_2$ and say that $V_1$ and $V_2$ are {\em directly linked by $X$}.  Direct linkage generates an equivalence relation called {\em liaison}.  If we further restrict to {\em even} numbers of links, we obtain the equivalence relation of {\em even liaison}.

\begin{remark}
In this paper we will always assume that our links are by complete intersections, as indicated above.  However, there is a beautiful and active field of {\em Gorenstein liaison}, where $X$ only needs to be arithmetically Gorenstein.  This allows for many additional results that are not true when $X$ must be a complete intersection, but it is also true that many interesting questions and results apply to this latter case, and this paper addresses some of these.  When both possibilities are considered, it is sometimes useful to distinguish between {\em G-liaison} and {\em CI-liaison}.  However, because of our restriction, we suppress the ``CI" here.
\end{remark}

It is well-known that $V$ is {\em arithmetically Cohen-Macaulay (ACM)} if and only if $H^i({\mathcal I}_V) = 0$ for all $1 \leq i \leq \dim V$.  More importantly, these modules are invariant up to shift in a non-ACM even liaison class:

\begin{theorem}[Hartshorne, Schenzel]
Let $V_1$ and $V_2$ be in the same {\bf even} liaison class, and assume that they are not ACM.  Then there is some integer $\delta$ such that
\[
H^i({\mathcal I}_{V_1}) \cong H^i({\mathcal I}_{V_2})(\delta) \hbox{\ \ for $i = 1,\dots,\dim V_1 = \dim V_2$}.
\]
\end{theorem}

We stress that the {\em same} $\delta$ is used for  all values of
$i$, so the intermediate cohomology modules move as a block under even
liaison.  It is not hard to see that there is a {\em leftmost} shift
of this block that can occur in an even liaison class, while {\em
any} rightward shift occurs \cite{BM4}.  We thus {\em partition} an
even liaison class $\mathcal L$ according to the shift of the block
of intermediate cohomology modules.  Knowing the exact value of the leftmost
shift is not needed -- its existence is guaranteed, and the
subschemes which achieve this leftmost shift are called the {\em
minimal elements} of the even liaison class.  In codimension two
there is a structure common to all even liaison classes, called the
{\em Lazarsfeld-Rao property} \cite{MDP,BBM,uwepaper,nollet},  which
is based on this partition.  It says, basically, that the minimal
elements are in fact also minimal with respect to the degree and the
arithmetic genus, and furthermore that the entire even liaison class
can be built up from any minimal element by a process called {\em
basic double linkage}, together with flat deformations that preserve
the block of cohomology modules.

At this point it is worth mentioning that, even for a licci
subscheme of codimension two, a single minimal link does not
necessarily produces a ``smaller'' residual subscheme.

\begin{example}
As an illustration, consider a set of points in $\mathbb P^2$ with
the following configuration:

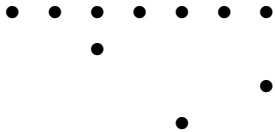
\begin{center}
\[
\begin{array}{ccccccccccccccc}
\bullet & \bullet & \bullet & \bullet & \bullet & \bullet & \bullet \\
&& \bullet \\
&&&&&& \bullet \\
&&&& \bullet
\end{array}
\]
\end{center}

\noindent (There are 7 points on a line, and 3 sufficiently general
points off the line.)  Gaeta's theorem applies since $Z$ is ACM of
codimension two, so minimal links do lead to a complete
intersection.  However, our main measure of minimality, the degree,
actually goes up with one minimal link.  Specifically, the smallest
link is with a regular sequence of type (3,7), so the residual set
of points actually has degree 11.  The next link, though, reduces to
a single point.
\end{example}

The example illustrates the philosophy that pairs of links are more
revealing than a single link and that one should study even liaison
classes. Note however that in higher codimension the situation seems
more complicated, even for Gorenstein liaison. In fact, a recent result of Hartshorne, Sabadini,
and Schlesinger \cite[Theorem 1.1]{HSS} says that a general set of
at least 56 points in $\PP^3$ does not admit a strictly descending
Gorenstein liaison or biliaison.
\smallskip

Since we will be using the construction of  basic double linkage
below, we briefly recall the important features for subschemes of
codimension two. There are several versions and extensions of basic
double linkage (for instance see \cite{GM4,KMMNP,HMN}), but here we
only require the version in codimension two, which is the simplest
case. Let $V \subset \mathbb P^n$ be an equidimensional codimension
2 subscheme. Let $f \in I_V$ be a homogeneous polynomial and let $\ell$
be a linear form such that $(\ell,f)$ is a regular sequence.  (It does
not matter if $\ell$ vanishes on a component of $V$ or not.)  Then the
ideal $\ell \cdot I_V + (f)$ is the {\em saturated} ideal of a
subscheme, $Y$, that is linked to $V$ in two steps.  In particular,
the block of modules for $Y$ is obtained from that of $V$ by
shifting one spot to the right.  It is this mechanism that
guarantees that all rightward shifts occur, as mentioned above.
Geometrically, $Y$ is the union of $V$ with the (degenerate)
complete intersection of $\ell$ and $f$.

Virtually all of the results in (complete intersection) liaison theory use, in one way or another, the so-called {\em Rao correspondence} between the even liaison classes and the stable equivalence classes of certain locally free sheaves \cite{rao1,rao2}.  We will need certain aspects of this correspondence, which we now recall.  Assume that $V \subset \mathbb P^n$ is a codimension two equidimensional {\em locally Cohen-Macaulay} subscheme.  Then the ideal sheaf ${\mathcal I}_V$ has {\em locally free} resolutions of the form
\[
\begin{array}{cccccccccccccccccccccc}
0 \rightarrow {\mathcal F_1} \rightarrow {\mathcal N} \rightarrow {\mathcal I}_V \rightarrow 0 & && \hbox{($N$-type resolution)} \\ \\
0 \rightarrow {\mathcal E} \rightarrow {\mathcal F_2} \rightarrow  {\mathcal I}_V \rightarrow 0 & && \hbox{($E$-type resolution)}
\end{array}
\]
where $H^1_*({\mathcal E}) = 0$ and $H^{n-1}_*({\mathcal N}) = 0$ and $\mathcal F_1$ and $\mathcal F_2$ are direct sums of line bundles on $\mathbb P^n$. Thus, taking global sections we get short exact sequences of graded $R$-modules
\[
\begin{array}{cccccccccccccccccccccc}
0 \rightarrow F_1 \rightarrow N \rightarrow I_V \rightarrow 0 & && \hbox{($N$-type resolution)} \\ \\
0 \rightarrow E \rightarrow F_2 \rightarrow  I_V \rightarrow 0 & && \hbox{($E$-type resolution)}
\end{array}
\]
where $F_1$ and $F_2$ are free $R$-modules.

We say that $\mathcal E_1$ and $\mathcal E_2$ (with the same hypotheses as $\mathcal E$) are {\em stably equivalent} if there exist direct sums of line bundles $\mathcal F_1, \mathcal F_2$ such that $\mathcal E_2 \oplus \mathcal F_2 \cong \mathcal E_1(c) \oplus \mathcal F_1$ for some integer $c$.

\begin{theorem} \cite{rao2}
Let $V_1, V_2 \subset \mathbb P^n$ be  locally Cohen-Macaulay  and
equidimensional subschemes and let $\mathcal N_1$ (resp.\ $\mathcal E_1$) and
$\mathcal N_2$ (resp.\ $\mathcal E_2$) be the locally free sheaves
appearing in their $N$-type (resp.\ $E$-type) resolutions.  Then
$V_1$ and $V_2$ are in the same even liaison class if and only if
$\mathcal N_1$ and $ \mathcal N_2$ (resp.\ $\mathcal E_1$ and
$\mathcal E_2$) are stably equivalent.

\end{theorem}

This theorem has been extended to the non-locally  Cohen-Macaulay
case \cite{uwepaper, nollet,hartshorne-toulouse} but we do not need
the extended version.  A corollary of this theorem is that $V$ (of
codimension two in $\mathbb P^n$) is {\em licci} if and only if it
is ACM.  See the introduction for an explanation of licciness, as
well as an improvement, due to Gaeta, involving minimal links.

Martin-Deschamps and Perrin \cite{MDP}  (for curves in $\mathbb
P^3$) and later Hartshorne \cite{hartshorne-toulouse} (for
codimension two subschemes of $\mathbb P^n$) gave a simplification
of the Lazarsfeld-Rao property by replacing the deformation with
constant cohomology by linear equivalence on a hypersurface.  This
is the notion of {\em biliaison}.  Now our sequence of basic double
links is replaced by a sequence of adding the class of a hyperplane
(or hypersurface) section in the Picard group of the hypersurface.
It is well known that if $V$ is linearly equivalent on a
hypersurface $F$ to a divisor $V + dH$, where $H$ is the class of a
hyperplane, then $V$ is linked in two steps to any element of the
linear system $|V+dH|$.  The difference between this and basic
double linkage is that in the latter case, we do not use linear
equivalence, but rather use one deformation at the end of the
sequence of adding to $V$ hyperplane sections of (possibly a
sequence of) hypersurfaces.

The  Lazarsfeld-Rao property using biliaison allows  the adding
{\em and subtracting} of the class of a hypersurface section on
hypersurfaces.  However, the notion of {\em strictly ascending
elementary biliaisons} refers to the fact that without loss of
generality, we can restrict to {\em adding} hypersurface sections.
In reverse, the notion in Theorem \ref{hart-thm} of the existence of
a {\em strictly decreasing elementary biliaison} refers to the
assertion that if $V$ is the divisor class of our subscheme on a
hypersurface $F$ then on $F$, $V-H$ is effective, where $H$ is the
class of a hyperplane.  Theorem \ref{hart-thm} asserts that whenever
$V$ is not minimal, there exists some hypersurface $F$ on which a
strictly decreasing elementary biliaison can take place.


\section{Minimal linkage does not necessarily give minimal elements in codimension 2} \label{codim 2}

In this section we give a class of examples to show that starting
with an arbitrary element of an even liaison class in codimension 2 and sequentially
applying minimal links, one does not necessarily arrive at a minimal
element of the even liaison class.

In order to produce such examples we use the correspondence of even
liaison classes and certain reflexive modules. We begin by defining
the modules we want to use. Let $R = k[x_0,\ldots,x_n]$, where $n \geq 3$, and consider
the artinian module $M := R/(x_0^{n+1},x_1,\ldots,x_n)$. Using its
minimal free resolution (given by the Koszul complex), we define the
modules $N_1$ and $N$ by
\[
\begin{array}{rcccl}
0 \to R(-2n-1) \to F_n \to F_{n-1} & & \to &  & F_{n-2} \to
\ldots \hspace*{3cm}\\
& \searrow && \nearrow \\
&& N_1 & \\
& \nearrow && \searrow \\
0 &&&&  0
\end{array}
\]

\[
\begin{array}{rcccl}
\hspace*{5cm} \cdots \to F_2 && \to &&  F_1 \to R \to M \to 0,  \\
 & \searrow && \nearrow \\
 && N & \\
 & \nearrow && \searrow \\
0 &&&&  0
\end{array}
\]
where
\[
F_{n-1} = R(-n+1)^n \oplus R(-2n+1)^{\binom{n}{2}} \quad \text{and}
\quad F_2 = R(-2)^{\binom{n}{2}} \oplus R(-n-2)^{n}.
\]
Next, we describe the minimal elements in the even liaison classes
$\cL_N$ and $\cL_{N_1}$ corresponding to $N$ and $N_1$, respectively.

\begin{lemma}
  \label{lem-min-el}
\begin{itemize}
  \item[(a)] Each minimal element, $I_{min}$, in the even liaison class
$\cL_N$
  has a minimal $N$-type resolution of the form
\begin{equation}
\label{eq-N-type-N}
0 \to R(-n+1)^{n-2} \oplus R(-2n+1) \to N(-n+3) \to I_{min} \to 0
\end{equation}
  and the shape of its minimal free resolution is
\[
0 \to R(-3n+2) \to F_{n}(-n+3) \to \cdots \to F_3(-n+3)
\to \hspace*{2.5cm}
\]
\[
\hspace*{5cm} R(-n+1)^{\binom{n-1}{2} +1} \oplus R(-2n+1)^{n-1}
\to I_{min} \to 0.
\]

\item[(b)] Each minimal element, $J_{min}$, in the  class $\cL_{N_1}$
  has a minimal $N$-type resolution of the form
\begin{equation}
\label{eq-N-type-E}
0 \to R(-n+1)^{n-2} \oplus R(-2n+1)^{\binom{n-1}{2}}
\to {N_1} \to J_{min} \to 0
\end{equation}
  and the shape of its minimal free resolution is
\[
0 \to R(-2n-1) \to
\begin{array}{c}
R(-n) \\
\oplus \\
R(-2n)^{n}
\end{array} \to
\begin{array}{c}
R(-n+1)^2 \\
\oplus \\
 R(-2n+1)^{n-1}
 \end{array}
\to J_{min} \to 0.
\]
\end{itemize}
\end{lemma}

\begin{proof}
  (a) Since $I_{min}$ is minimal in its even liaison class, its
  minimal $N$-type resolution has the form (see \cite{MDP} or
  \cite{uwepaper})
\[
0 \to F \to N \to I_{min} (c) \to 0,
\]
where $F$ is a free $R$-module such that the sum of the degrees of
its minimal generators is minimal. The module $N$ has rank $n$. Thus
the module $F$ is determined by a sequence of homogeneous elements
$m_1,\ldots,m_{n-1} \in N$ of least degree that is $(n-1)$-fold
basic in $N$ at prime ideals of codimension at most one (see
\cite{uwepaper}, Proposition 6.3). Notice that the module $N$ has
$\binom{n}{2}$ minimal generators of degree two and $n$ generators
of degree $n+2$. Moreover, the Koszul complex that resolves $M$
shows that the generators of degree two generate a torsion-free
$R$-module of rank $n-1$. Hence, we can find $n-2$ elements in $N$
of degree two that form an $(n-2)$-fold basic sequence in $N$. Since
besides the generators of degree two, $N$ has only generators in
degree $n+2$, this must be the degree of the $(n-1)$-st element in
the $(n-1)$-fold basic sequence in $N$ of least degrees. This means,
the $N$-type resolution of $I_{min}$ is of the form
\[
0 \to R(-2)^{n-2} \oplus R(-n-2) \to N \to I_{min} (c) \to 0.
\]
The integer $c$ is determined by the first Chern classes of $F$ and
$N$. Using the exact sequence
\[
0 \to N \to F_1 \to R \to M \to 0
\]
we get $c_1(N) = c_1(F_1) = -2n-1$. It follows that
\begin{eqnarray*}
  c & = & c_1(N) - c_1 (F) \\
  & = & -2n-1 - [-2 (n-2) - n - 2] = n - 3.
\end{eqnarray*}
This establishes our assertion about the $N$-type resolution of
$I_{min}$. The claim about its minimal free resolution follows by
using the mapping cone procedure, the minimal free resolution of
$N$, and by observing that all the direct summands $R(-2)^{n-2} \oplus R(-n-2)$ of $F$
will split because the corresponding generators map onto minimal generators of
$N$.

(b) This is shown similarly. We only highlight the differences to
the proof of part (a). This time $J_{min}$ has a minimal $N$-type
resolution of the form
\[
0 \to G \to N_1 \to J_{min} (c') \to 0,
\]
where $G$ is a free $R$-module of rank $\binom{n}{2} - 1$ with
generators of least degrees. The minimal generators of $N_1$ have
degrees $n-1$ and $2n-1$. Since the generators of degree $n-1$
generate a torsion-free module of rank $n-1$, it follows that $G
\cong R(-n+1)^{n-2} \oplus R(-2n+1)^{\binom{n-1}{2}}$.

Using the defining sequence of $N_1$, a Chern class computation
provides $c_1(N_1) = - \binom{n}{2}(2n-1) + (n+1)(n-1)$. Thus, we get
\begin{eqnarray*}
  c' & = & c_1(N_1) - c_1 (G) \\
  & = & - \binom{n}{2}(2n-1) + (n+1)(n-1) - [-(n-2)(n-1) -
  \binom{n-1}{2}(2n-1)]\\
  &  = &  0.
\end{eqnarray*}
This proves the assertion about the $N$-type resolution of
$J_{min}$. We get its minimal free resolution by using the mapping
cone and observing the cancelation all of the direct summands
$R(-n+1)^{n-2} \oplus R(-2n+1)^{\binom{n-1}{2}}$ of $G$.
\end{proof}

\begin{corollary} \label{linked}
The even liaison classes $\mathcal L_N$ and $\mathcal L_{N_1}$ are residual classes.
\end{corollary}

\begin{proof}
It suffices to show that $I_{min}$ can be linked to a minimal element in $\cL_{N_1}$. To this end we compare $E$-type resolutions.

The minimal free resolution of $I_{min}$ shows that we can link it by a complete intersection of type $(n-1, 2n-1)$. Using the mapping cone procedure, we see that the residual $J$ has a minimal $E$-type resolution
\[
0 \rightarrow N^*(-2n-1) \rightarrow
\begin{array}{c}
R(-n+1)^2 \\
\oplus \\
R(-2n+1)^{n-1}
\end{array}
\rightarrow J \rightarrow 0.
\]
On the other hand, the self-duality of the resolution of $M$ provides the short exact sequence \[
0 \to N^*(-2n-1) \to R(-n+1)^n \oplus R(-2n+1)^{\binom{n}{2}} \to N_1 \to 0.
\]
Applying the mapping cone procedure to the $N$-type resolution of $J_{min}$ and cancelling redundant terms, we get as its $E$-type resolution
\[
0 \rightarrow N^*(-2n-1) \rightarrow
\begin{array}{c}
R(-n+1)^2 \\
\oplus \\
R(-2n+1)^{n-1}
\end{array}
\rightarrow J_{min} \rightarrow 0.
\]
Comparing with the above $E$-type resolution of $J$, we see that $J$ is a minimal element in $\cL_{N_1}$, as claimed.
\end{proof}

We are ready for the main result of this section. Recall that a subscheme $X \subset \PP^n$ is said to be {\em non-degenerate} if $X$ is not a subscheme of any hyperplane of $\PP^n$. Equivalently, this means that the homogeneous ideal $I_X$ of $X$ does not contain any linear form.

\begin{theorem}
  \label{thm-lcm-non-min-licci}
Let $f \in I_{min}$ be a form of degree $n+1$ and let $\ell \in R$
be a linear form such that $\ell, f$ is an $R$-regular sequence. Let
$X \subset \PP^n$ be the codimension two subscheme defined by the
saturated homogeneous ideal
\[
I_X = \ell \cdot I_{min} + (f).
\]
Then $X$ is a non-degenerate, locally Cohen-Macaulay subscheme and $I_X$ is in the even liaison
class $\cL_N$ of $I_{min}$, but $I_X$ cannot be linked to a minimal
element in $\cL_N$ by way of minimal links.
\end{theorem}

\begin{proof}
By definition $I_X$ is a basic double link of $I_{min}$. Thus,
  both ideals are in the same even liaison class. Moreover, the
cohomology of
  $X$ is determined by those of $N$. More precisely, one has
\[
H^i_* (\PP^n, \cI_X) \cong \left \{
\begin{array}{cl}
  M & \text{if }  i = 1 \\
  0 & \text{if } 2 \leq i \leq n-2
\end{array}. \right.
\]
Hence $X$ is locally Cohen-Macaulay.

Since $I_X$ is defined by a basic double link, there is
  an exact sequence
\begin{equation*}
  0 \to R(-n-2) \to R(-n-1) \oplus I_{min} (-1) \to I_X \to 0
\end{equation*}
(cf.\ \cite{KMMNP}). 
Thus, the mapping cone procedure provides that $I_X$ has a (not
necessarily minimal) free resolution of the form
\begin{equation}
  \label{eq-res-I}
0 \to R(-3n+1) \to \cdots \to
\begin{array}{c}
R(-n)^{\binom{n-1}{2} +1} \\
\oplus \\
 R(-n-1) \\
 \oplus \\
 R(-2n)^{n-1}
\end{array}  \to
I_{X} \to 0
\end{equation}
and an $N$-type resolution of the form
\begin{equation}
\label{eq-N-type-I}
0 \to
\begin{array}{c}
R(-n)^{n-2} \\
\oplus \\
R(-n-2) \\
\oplus \\
R(-2n)
\end{array}
\to
\begin{array}{c}
  R(-n-1) \\
  \oplus \\
 N(-n+2)
 \end{array}
\to I_{X} \to 0.
\end{equation}
Now, let $I$ be any ideal which has a minimal free resolution and an
$N$-type resolution as the ideal $I_X$ above. We want to show:
\smallskip

{\em Claim 1}: $I$ does not contain a regular sequence of type $(n,
n+1)$.
\smallskip

Indeed, suppose this were not true. Then, we link $I$ to an ideal
$J'$ by a complete intersection of type $(n, n+1)$. The residual
$J'$ has an $E$-type resolution of the form (see, e.g.,
\cite{uwepaper}, Proposition 3.8)
\begin{equation*}
  0 \to
\begin{array}{c}
R(-n) \\
\oplus \\
N^{*}(-n-3)
\end{array}
\to
\begin{array}{c}
R(-1)
  \oplus R(-n+1) \\
  \oplus \\
R(-n) \oplus  R(-n-1)^{n-1}
  \end{array} \to J' \to 0.
\end{equation*}
It
follows that the minimal free resolution of $J'$ ends with a module
generated in degree $n+3$, i.e., it is of the form
\[
0 \to R(-n-3) \to \cdots \to J' \to 0.
\]
However, thanks to Corollary \ref{linked}, $J'$ is in the even liaison class $\cL_{N_1}$ and comparing with
the minimal free resolution of  $J_{min}$, this gives a
contradiction to the minimality of $J_{min}$ because $n+3 < 2n+1$.
Thus, Claim 1 is established.
\smallskip

 From the free resolution of $I$ we see that the degrees of the
minimal generators of $I$ are at most $n, n+1$, and $2n$, and that there
is at least one minimal generator of degree $n$.  Hence Claim 1
implies that the complete intersection of least degree inside $I$
has type $(n, 2n)$. Let $J$ be the residual of $I$ with respect to
this complete intersection. Its $E$-type resolution is of the form
\begin{equation}
\label{eq-E-type-J}
0 \to
\begin{array}{c}
R(-2n+1) \\
\oplus \\
N^*(-2n-2)
\end{array}
\to
\begin{array}{c}
  R(-n)^2 \\
  \oplus \\
 R(-2n+2) \\
 \oplus \\
 R(-2n)^{n-1}
 \end{array}
\to J \to 0.
\end{equation}
We want to find the minimal link of $J$.
\smallskip

{\em Claim 2}: $J$ does not contain a complete intersection of type
$(n, 2n-2)$.
\smallskip

Indeed, otherwise we could link $J$ by a complete intersection of
type $(n, 2n-2)$ to an ideal $I'$ with $N$-type resolution
\begin{equation*}
\label{eq-N-type-I'}
0 \to
\begin{array}{c}
R(-n+2)^{n-1} \\
\oplus \\
R(-n) \\
\oplus \\
R(-2n+2)^2
\end{array}
\to
\begin{array}{c}
  R(-n+1) \oplus R(-n) \\
  \oplus \\
  R(-2n+2)  \\
  \oplus      \\
 N(-n+4)
 \end{array}
\to I' \to 0.
\end{equation*}
Comparing the degree shift of $N$ in the $N$-type resolution of $I'$
with those in the resolution of $I_{min}$, we get a contradiction to
the minimality of $I_{min}$ because $n-4 < n-3$. This establishes
Claim 2.
\smallskip

The free resolution (\ref{eq-E-type-J}) shows that the ideal $J$ is
generated in degrees $n, 2n -2$, and $2n$. Since $J$ has at most one
generator of degree $2n-2$, the claim implies that the minimal link
of $J$ is given by a complete intersection of type $(n, 2n)$. Denote
the residual by $\tilde{I}$. Using the cancelation due to the fact
that the minimal generators of the complete intersection are also
minimal generators of $J$, we see that $\tilde{I}$ has an $N$-type
resolution of the form
\begin{equation*}
\label{eq-N-type-II}
0 \to
\begin{array}{c}
R(-n)^{n-2} \\
\oplus \\
R(-n-2) \\
\oplus \\
R(-2n)
\end{array}
\to
\begin{array}{c}
  R(-n-1) \\
  \oplus \\
 N(-n+2)
 \end{array}
\to \tilde{I} \to 0.
\end{equation*}
Since the generators of $R(-n)^{n-2}$ map onto minimal generators of
$N(-n+2)$, the mapping cone procedure provides (after taking into
account the resulting cancelation) that $\tilde{I}$ has a free
resolution of the form
\begin{equation*}
  \label{eq-res-III}
0 \to R(-3n+1) \to \cdots \to
\begin{array}{c}
R(-n)^{\binom{n-1}{2} +1} \\
\oplus \\
 R(-n-1) \\
 \oplus \\
 R(-2n)^{n-1}
\end{array}  \to
\tilde{I} \to 0.
\end{equation*}
Comparing this with resolutions (\ref{eq-res-I}) and
(\ref{eq-N-type-I}), we see that any ideal $I$ with $N$-type and
minimal free resolution as $I_X$ is minimally linked in two steps to
an ideal with the same resolutions. Hence $I$ (and also $J$) cannot
be minimally linked to a minimal element in its even liaison class.
\end{proof}

\begin{remark}
  (i) The arguments in the above theorem can easily be modified to
  provide other examples of subschemes that cannot be minimally linked to
  a minimal element in its even liaison class. Indeed, instead of
  starting with the module $M := R/(x_0^{n+1},x_1,\ldots,x_n)$, one
  could use $R/(x_0^{d},x_1,\ldots,x_n)$ for some $d \geq n+1$.
  Defining the modules $N_1$ and $N$ analogously, one can then take
  the basic double link of a minimal element in the class $\cL_N$ on
  a hypersurface of suitable degree to get further examples with a the
  same behavior as $X$. However, we restricted ourselves to the
  specific choices above in order to keep the arguments as simple as
  possible.

  (ii) Further examples can be obtained by taking cones. Notice
  however, that contrary to the examples we constructed, these cones
  are no longer locally Cohen-Macaulay.
\end{remark}

In the case of curves in $\mathbb P^3$, simpler constructions and arguments suffice to make our conclusions.  Omitting the proofs (which are simplifications of the ones given above), the following can be shown.

\begin{proposition} \label{facts about C1}
Let $C_1$ be a curve in $\mathbb P^3$ consisting of the disjoint
union of a line and a plane curve, $Y$, of degree $d$, with $d \geq
1$.  Then

\begin{itemize}
\item[(a)] $M(C_1) \cong k[w]/(g)$, where $g(w)$ is a polynomial of
degree $d$.  In particular,
\[
\dim M(C)_t =
\left \{
\begin{array}{ll}
1 & \hbox{if } 0 \leq t \leq d-1; \\
0 & \hbox{otherwise.}
\end{array}
\right.
\]

\item[(b)] $C_1$ is minimal in its even liaison class.

\item[(c)] The minimal free resolution of $I_{C_1}$ is of the form
\[
0 \rightarrow R(-d-3) \stackrel{\sigma}{\longrightarrow}
\begin{array}{c}
R(-3) \\
\oplus \\
R(-d-2)^3
\end{array}
\rightarrow
\begin{array}{c}
R(-d-1)^2 \\
\oplus \\
R(-2)^2
\end{array}
\rightarrow I_{C_1} \rightarrow 0.
\]

\item[(d)] $M(C_1)$ is self-dual, hence the even liaison coincides with
the entire liaison class.  In particular, two curves that are directly
linked are also evenly linked.

\item[(e)] The minimal link for $C_1$ is a complete intersection of type
$(2,d+1)$.  Such a complete intersection links $C_1$ to another minimal
curve in the (even) liaison class.

\end{itemize}

\end{proposition}

\begin{theorem} \label{no min link to min}
Choose integers $d,e$ satisfying
\[
4 \leq e \leq d.
\]
Let $C_1$ be the disjoint union of a line and a plane curve, $Y$, of
degree $d$.  Let $f$ be a homogeneous  element of $I_{C_1}$ of degree
$e$.  Observe that $f$ necessarily contains as a factor the linear form
of $I_Y$.  Let $\ell$ be a general linear form, and define $C$ by the
saturated homogeneous ideal
\[
I_C = \ell \cdot I_{C_1} + (f).
\]
The curve $C$ cannot be linked to a minimal curve in its even liaison
class by way of minimal links.
\end{theorem}

\begin{remark}
Since minimal curves in any even  liaison class are minimally linked
to minimal curves in the residual liaison class (cf.\ \cite{MDP},
Theorem IV.5.10), it follows from Theorem~\ref{no min link to min}
that $C$ is not minimally linked to a minimal curve in the residual
class either.  (In any case, this module is self-dual so the minimal
curves coincide.)
\end{remark}

\begin{remark}
In \cite{HUduke}, Huneke and Ulrich showed that if an ideal $I$  in
a local Gorenstein ring with infinite residue field is licci, then
one can pick regular sequences consisting of general linear
combinations of minimal generators at every linkage step to reach a
complete intersection. Gaeta's theorem may be viewed as an analogous
result for projective ACM subschemes of codimension two. Theorem
\ref{thm-lcm-non-min-licci} shows that, in general, its conclusion
is not longer true if one drops the assumption that the schemes are
ACM.  
\end{remark}

\begin{remark}
It is not true that if a curve is irreducible then it is  minimally
linked to a minimal curve.  Indeed, the curve $C$ produced in
Theorem \ref{no min link to min} can be linked using two generally
chosen surfaces of degree 12 to a residual curve, $C'$, which is
smooth and for which the smallest complete intersection is again of
type (12,12).  Hence $C'$ is minimally linked to $C$, and we have
seen that $C$ is not minimally linked to a minimal curve.  However,
it is an open question whether we can find an irreducible curve $C$
that is minimally linked in two steps back to a curve that is
numerically the same as $C$ in the sense of Theorem \ref{no min link
to min}.
\end{remark}

\begin{remark}
We now make the connection between Hartshorne's  Theorem
\ref{hart-thm} and our results above, and in particular why the
latter do not contradict the former.  For simplicity we will use the
context of curves in $\mathbb P^3$, and Theorem \ref{no min link to
min}, for our discussion, but it holds equally well in the
codimension two setting.    It is clear (as Hartshorne's theorem
guarantees) that $C$ admits a strictly descending elementary
biliaison.  It does so on the surface $F$ defined by $f$, and one
obtains the minimal curve $C_1$ as a result: $C-H = C_1$ on $F$,
where $H$ is the class of a hyperplane section of $F$. But this only guarantees a pair of links of the form $(f,a)$ and $(f,b)$, where $a$ and $b$ are forms and $\deg b < \deg a$.  One can show that  $C$ does not admit a link of
type $(3, d)$. Using the assumption $4 \leq e \leq d$, it follows
that neither $f$ nor $a$ can have degree 3. Hence Hartshorne's
theorem says nothing about the minimal link of $C$ in this case,
since a minimal link necessarily involves a surface of degree 3.

What about the possibility that in addition to  the elementary
biliaison mentioned above there is another one that involves a
minimal link as the first step?  Recall that a minimal link for $C$
is of type $(3,d+2)$.  Suppose that there is a strictly descending
elementary biliaison on a surface $G$, with $\deg G = d+2$ or $\deg
G = 3$, where the first link is minimal.  That means, in either
case, that after performing the first link (as was done in the
proof) one obtains a residual that allows a smaller link.  This is
precisely what was proved to be impossible in Theorem \ref{no min
link to min}.
\end{remark}


\section{Hypersurface sections}

In this section we show that in any $\mathbb P^n$ ($n \geq 4$) there exist non-ACM, locally Cohen-Macaulay subschemes of any codimension $c$ with $2 \leq c \leq n-1$ for which no sequence of minimal links reaches a minimal element of the even liaison class.  The last section showed this result when $c=2$, so this will be our starting point.  Our method of attack will be via hypersurface sections of large degree, in a manner analogous to that used in \cite{HMNU} for the licci class, although here some new ideas are needed.  Since non-ACM schemes have at least one non-zero intermediate cohomology  module, our measure of minimality will be by the shift of the collection of intermediate cohomology  modules (see section \ref{liaison facts}).

If $X$ is at least a surface, we will show that under suitable conditions we can take hypersurface sections and preserve the linkage property.  So let $X \subset \mathbb P^n$ be a non-ACM, codimension $c-1$ subscheme with $\dim X \geq 2$ and satisfying the following properties:

\begin{itemize}
\item[(i)] $X$ is not minimal in its even liaison class.

\item[(ii)] $X$ has the property that no sequence of minimal links arrives at a minimal element of the even liaison class (and hence, as noted in the last section, no sequence of minimal links arrives at a minimal element of the residual even liaison class either).

\item[(iii)] $H^1_*(\mathbb P^n,{\mathcal I}_X) \cong (R/I)(-\delta)$  for some artinian ideal $I$ and some $\delta > 0$.  Let $e$ be the socle degree of $R/I$, i.e. the last degree in which $R/I$ is non-zero.

\item[(iv)] $H^i_*(\mathbb P^n, {\mathcal I}_X) = 0$ for all $2 \leq i \leq \dim X$.

\end{itemize}
Note that $X$ is automatically locally Cohen-Macaulay and equidimensional, thanks to the conditions on the cohomology of $\mathcal I_X$.

\begin{proposition}\label{hypersurf}
Let $X$ be as above.  Fix an integer $d > e$ (see (iii)) and  assume
also that $d+\delta$ is greater than any degree of a minimal
generator of $I_X$. Let $F$ be a general form of degree $d$.  Let
$Y$ be the hypersurface section of $X$ cut out by $F$, so $I_Y$ is
the saturation of $(I_X,F)$.  Then $Y$ also satisfies the property
that no sequence of minimal links starting with $Y$ arrives at a
minimal element of the liaison class.
\end{proposition}

\begin{proof}
We have taken $F$ to be a general form of  degree $d$.  The meaning
of ``general'' will be made more precise as we go through the proof,
but at each step it will be clear that each new constraint is still
an open condition, and there are finitely many of them.  The first
condition, of course, is that $F$ meets each component of $X$
properly.

Note first that since $H^1_*(\mathbb P^n, {\mathcal I}_X) \neq 0$,
$R/I_X$ has depth 1, and the last free module in the minimal free
resolution of $R/I_X$ is the same as that in the minimal free
resolution of $H^1_*(\mathbb P^n, {\mathcal I}_X)$.  The twist of
this last free module measures the shift of the intermediate cohomology
module, and hence measures how far $X$ is from being minimal in its
even liaison class.

We first claim that $I_Y$ has two more generators than $I_X$ does: one is $F$, and the other comes from the intermediate cohomology module.  As an intermediate step, we have to consider the ideal of $Y$ in the hypersurface defined by $F$, i.e. $I_{Y|F}$.  We use some ideas from \cite{migbook}.

Consider the exact sequence of sheaves,
\begin{equation} \label{ses}
0 \rightarrow {\mathcal I}_X(-d) \stackrel{\times F}{\longrightarrow} {\mathcal I}_X \rightarrow {\mathcal I}_{Y|F} \rightarrow 0.
\end{equation}
Since $d \geq e$, multiplication by $F$ is zero on $R/I$.   Hence we
have the following long exact sequence in cohomology (taking direct
sums over all twists):
\begin{equation} \label{les}
\begin{array}{ccccccccccccccccccccccccccc}
0 & \rightarrow & I_X(-d) & \stackrel{\times F}{\longrightarrow} & \hspace{-.4cm} I_X &  \longrightarrow & I_{Y|F} & \rightarrow (R/I)(-\delta - d) \rightarrow 0\\
&&&& \hbox{\hspace{.6cm}} \searrow && \nearrow \hspace{.6cm} \\
&&&&& \frac{I_X}{F \cdot I_X} \\
&&&& \hbox{\hspace{.6cm}} \nearrow && \searrow \hspace{.6cm} \\
&&&& \hspace{-.4cm} 0 && \hspace{.4cm} 0
\end{array}
\end{equation}
because $F$ annihilates $R/I$  by the assumption about its degree.
Since $\frac{I_X}{F \cdot I_X}$ has the same degrees of minimal generators as $I_X$ and $R/I$ has one minimal generator, the Horseshoe Lemma applied to the last short exact sequence coming from (\ref{les})  shows that $I_{Y|X}$ has exactly one extra generator (besides those coming from the restriction of $I_X$ to $F$), coming in degree $d+\delta$.  The embedding $(F) \hookrightarrow I_Y$ induces the short exact sequence
\[
0 \rightarrow R(-d) \stackrel{\times F}{\longrightarrow} I_Y \rightarrow I_{Y|F} \rightarrow 0,
\]
which shows that $I_Y$ has one additional minimal  generator, namely
$F$ itself.  To conclude:
\begin{quote}
{\em One minimal generator of $I_Y$ is $F$, and one has degree $d + \delta$.  The remaining degrees of minimal generators of $I_Y$ coincide with  those of $I_X$, and in any minimal generating set for $I_{Y|F}$, the  generators of $I_{Y|F}$ corresponding to these remaining degrees all  lift to minimal generators of $I_X$ (by degree considerations). }
\end{quote}

Note that the above observations apply to any $X'$ in  the same even
liaison class, and with the same degrees of generators, as $X$.

Recall that $F$ was chosen to be a ``general'' form of degree  $d$;
we now add a requirement for this generality.  We have assumed that
$X$ is not minimal in its even liaison class, which means that $X$
can be linked in some sequence of (non-minimal) links down to one
which has a more leftward shift of the deficiency modules.  We
assume that $F$ meets all of these intermediate links, including the
linking hypersurfaces, properly.  It follows that $Y$ is not minimal
in its even liaison class: simply adjoin $F$ to each of the complete
intersections used in the sequence of links for $X$, to obtain a
subscheme which has a more leftward shift of the deficiency modules
than $Y$. Here we have used again the fact that $F$ annihilates all
the intermediate cohomology modules of the subschemes participating
in the links.


Now suppose that $(G_1,\dots,G_c)$ is a minimal complete
intersection in $I_Y$. In particular, we may take the $G_i$ to be part of a minimal generating set for $I_Y$.  Since $I_X + (F)$ is generated in degrees $<
d + \delta$, all of these $G_i$ have degrees $< d+\delta$.
Furthermore, without loss of generality we may take one of them, say
$G_c$, to be $F$.  As for  the rest, if we consider their restriction to $I_{Y|F}$, they  necessarily lift to $I_X$.  Let
$Y'$ be defined by $I_{Y'} = (G_1,\dots, G_c) : I_Y$.  Say
$(G_1,\dots,G_{c-1}) \subset I_Y$ restricts to $(\bar G_1,\dots,\bar G_{c-1}) \subset I_{Y|F}$, which in turn  lifts to $(F_1,\dots,F_{c-1})
\subset I_X$.  Clearly this is a minimal link for $I_X$.  Consider
the residual subscheme $X'$ defined by $I_{X'} = (F_1,\dots,F_{c-1})
: I_X$.  By the Hartshorne-Schenzel theorem, $H^i_*({\mathcal
I}_{X'}) = 0$ for $1 \leq i \leq n-3$.  In particular, $R/I_{X'}$
has depth $\geq 2$.  Hence $(I_{X'},F)$ is saturated, and   $I_{Y'}
= (I_{X'},F)$.  Then any minimal generating set of $I_{Y'}$ can
(without loss of generality) be written as $F$ together with the
restrictions of minimal generators of $I_{X'}$.

It follows that a minimal link for $I_{Y'}$ lifts to a  minimal link
for $X'$.  By the numerical conditions that we have set, any finite
sequence of minimal links starting with $Y$ will lift to a sequence
of minimal links starting with $X$.  Since the latter never allow
for a smaller link, the same is true for the former.
\end{proof}

Combining with the main result of section \ref{codim 2} we obtain

\begin{corollary}
In any codimension there are locally Cohen-Macaulay, equidimensional, non-ACM subschemes that are not minimal in their even liaison class, and cannot be minimally linked in a finite number of steps to a minimal element.
\end{corollary}

\begin{proof}
We simply observe that from the exact sequence (\ref{ses}) it also follows that $Y$ satisfies properties (iii) and (iv) before Proposition \ref{hypersurf}.  Properties (i) and (ii) come directly from Proposition \ref{hypersurf}, so we may successively apply that proposition as many times as we like as long as the scheme that we are cutting with a hypersurface is at least of dimension 2.  To start the process we have the main result of section \ref{codim 2}.
\end{proof}

Of course this is not surprising, given the main result of \cite{HMNU}, but still an example needed to be found.


\section{A non-arithmetically Cohen-Macaulay extension of Gaeta's theorem}
\label{2skewlines}

In view of Gaeta's result on the one hand and Theorem \ref{no min link to min} on the other, it is natural to ask if the arithmetically Cohen-Macaulay case is the only one where every curve is minimally linked to a minimal curve.  In this section we show that this is not the case!

For the following theorem we will consider the liaison class of curves in $\mathbb P^3$ whose corresponding intermediate cohomology module is isomorphic to $k^n$, concentrated in one degree.  We will call this liaison class $\mathcal L_n$.  Note that since the module  $k^n$ is self-dual, any two curves in $\mathcal L_n$ are both evenly and oddly linked.   The following provides some facts that we will need.

\begin{lemma} \label{min facts}
Let $C \in \mathcal L_n$.  The following are equivalent.

\begin{enumerate}
\item \label{min} $C$ is a minimal curve in $\mathcal L_n$.

\item \label{deg} $\deg C = 2n^2$.

\item \label{shift} $M(C)$ is non-zero in degree $2n-2$.

\item \label{mfr} $I_C$ has a minimal free resolution
\[
0 \rightarrow R(-2n-2)^n \rightarrow R(-2n-1)^{4n} \rightarrow R(-2n)^{3n+1} \rightarrow I_C \rightarrow 0.
\]
\end{enumerate}
\end{lemma}

\begin{proof}
Parts (\ref{min}), (\ref{deg}) and (\ref{shift}) are from \cite{BM1}, also applying the Lazarsfeld-Rao property \cite{MDP}, \cite{BBM}, while (\ref{mfr}) is an easy calculation using \cite{rao1}.
\end{proof}

\begin{theorem} \label{minlinkLn}
Let $C$ be a curve in the liaison class $\mathcal L_n$.  Then $C$ is minimally linked to a minimal curve (in a finite number of steps).
\end{theorem}

\begin{proof}
The graded module $k^n$ has minimal free resolution
{\small
\begin{equation} \label{mfr of k}
\begin{array}{cccccccccccccccccccc}
0 & \rightarrow & R(-4)^n & \stackrel{\sigma}{\longrightarrow} & R(-3)^{4n} & \rightarrow & R(-2)^{6n} & \rightarrow & R(-1)^{4n} & \rightarrow & R^n & \rightarrow & k^n & \rightarrow & 0. \\
&&&& \hfill \searrow && \nearrow \hfill \searrow && \nearrow \hfill \\
&&&&& E && E^* (-4) \\
&&&& \hfill \nearrow && \searrow \hfill \nearrow && \searrow \hfill \\
&&&& 0 && 0 && 0
\end{array}
\end{equation} }
\hspace{-.3cm} Let $a$ be the initial degree of $I_C$ and let $b$ be the first degree in which $I_C$ has a regular sequence of length two.  Thanks to a theorem of Rao \cite{rao1}, $I_C$ has a minimal free resolution of the form
{\small
\begin{equation} \label{mfr of IC}
\begin{array}{cccccccccccccccccc}
0 & \rightarrow & R(-e)^n & \stackrel{(\sigma,0)}{\longrightarrow} &
\begin{array}{c}
R(-e+1)^{4n} \\
\oplus \\
{\mathbb F}
\end{array}
& \longrightarrow &
\begin{array}{c}
R(-a) \\
\oplus \\
R(-b) \\
\oplus \\
{\mathbb G}
\end{array}
& \rightarrow & I_C & \rightarrow & 0 \\
&&&& \hfill \searrow && \nearrow \hfill \\
&&&&&
\begin{array}{c}
E(4-e) \\
\oplus \\
{\mathbb F}
\end{array} \\
&&&& \hfill \nearrow && \searrow \hfill \\
&&&& 0 && 0
\end{array}
\end{equation}}
\hspace{-.3cm} where $\mathbb F$ and $\mathbb G$ are free modules, and $\sigma$ is the homomorphism coming from the minimal free resolution of $k^n$.
Note that the intermediate cohomology module, $M(C)$, of $C$ occurs in degree $e-4$.

The strategy of the proof is as follows.  Note that as usual, if $C$ is linked by a complete intersection of type $(a,b)$ then the smallest complete intersection containing the residual, $C'$, is at most of type $(a,b)$.  If the assertion of the theorem is false, then we would eventually come to a curve $C$ for which the smallest complete intersection containing the residual, $C'$, is again of type $(a,b)$.  {\em We will show that if this situation arises then $C$ is already a minimal curve, and $a = b = 2$.}  To do this, we look for numerical conditions that guarantee that $C'$ allows a smaller link, and rule out these numerical conditions.  This is restrictive enough that it forces $C$ to be minimal.  There are two such kinds of numerical conditions that we will use:  (i) all the minimal generators of $I_{C'}$ occur in degree $<b$, and (ii) the initial degree of $I_{C'}$ is $<a$.

We first find a free resolution of the residual curve, $I_{C'}$. Using the $E$-type resolution of $I_C$ indicated in (\ref{mfr of IC}) and splitting the summands corresponding to $R(-a)$ and $R(-b)$, we see that $I_{C'}$ has an $N$-type resolution of the form
\[
0 \rightarrow {\mathbb G}^*(-a-b) \rightarrow
\begin{array}{c}
E^* (e-4-a-b) \\
\oplus \\
{\mathbb F}^* (-a-b)
\end{array}
\rightarrow I_{C'} \rightarrow 0.
\]
From (\ref{mfr of k}) we know a minimal free resolution of $E^*$.  Hence an application of the mapping cone gives a free resolution (not necessarily minimal) of $I_{C'}$:
\begin{equation} \label{free resol of C'}
0 \rightarrow R(-4+e-a-b)^n \stackrel{(\sigma,0)}{\longrightarrow}
\begin{array}{c}
R(-3+e-a-b)^{4n} \\
\oplus \\
{\mathbb G}^* (-a-b)
\end{array}
\rightarrow
\begin{array}{c}
R(-2+e-a-b)^{6n} \\
\oplus \\
{\mathbb F}^*(-a-b)
\end{array}
\rightarrow I_{C'} \rightarrow 0.
\end{equation}

Now, assume that $\mathbb F \neq 0$ and let $\mathbb F = \bigoplus_{i = 1}^r R(-c_i)$.  Set
\[
\begin{array}{rcl}
c & = & \max \{ c_i \},  \\ \\
d & = & \min \{ c_i \}.
\end{array}
\]
Since the smallest twist of the free modules in the minimal free resolution of $I_C$ (\ref{mfr of IC}) is strictly increasing, we obtain
\begin{equation} \label{first ineq}
\begin{array}{rcl}
e & \geq & a+2, \\ \\
d & \geq & a+1
\end{array}
\end{equation}
Since $E^*(e-4-a-b)$ is not a summand of $\mathbb G^*(-a-b)$, not all summands of $R(-2+e-a-b)^{6n}$ split off in (\ref{free resol of C'}).  In particular,
\[
\hbox{$I_{C'}$ has at least one generator of degree $a+b+2-e$}.
\]
  Also, no summand of $\mathbb F^* (-a-b)$ splits off. Still assuming $\mathbb F \neq 0$, it follows that the degrees of the generators of $I_{C'}$ corresponding to $\mathbb F^*(-a-b)$ range from $a+b-c$ to $a+b-d$.  In particular, using (\ref{first ineq}), we have that the generators corresponding to $\mathbb F^*(-a-b)$ (if there are any) all have degree $\leq b-1$.

Suppose that $e > a+2$.  Then $a+b+2-e < b$.  Combining with the previous paragraph, we see that $I_{C'}$ has no minimal generator of degree $\geq b$, and this is also true if $\mathbb F$ is trivial.  It follows that $C'$ allows a smaller link.  Thus combining with (\ref{first ineq}), we obtain that
\[
\hbox{without loss of generality we may assume $e = a+2$.}
\]

Now we again consider the minimal free resolution (\ref{mfr of IC}), as well as the free resolution of $I_{C'}$, using this new substitution:
\begin{equation} \label{new mfr of IC}
0 \rightarrow R(-a-2)^n \rightarrow
\begin{array}{c}
R(-a-1)^{4n} \\
\oplus \\
\mathbb F
\end{array}
\rightarrow
\begin{array}{c}
R(-a) \\
\oplus \\
R(-b) \\
\oplus \\
\mathbb G
\end{array}
\rightarrow I_C \rightarrow 0
\end{equation}
and
\begin{equation} \label{new resol of IC'}
0 \rightarrow R(-b-2)^n \rightarrow
\begin{array}{c}
R(-b-1)^{4n} \\
\oplus \\
\mathbb G^*(-a-b)
\end{array}
\rightarrow
\begin{array}{c}
R(-b)^{6n} \\
\oplus \\
\mathbb F^* (-a-b)
\end{array}
\rightarrow I_{C'} \rightarrow 0.
\end{equation}

Recall that $\mathbb F = \bigoplus_{i=1}^r R(-c_i)$ (or 0) and $c = \max \{ c_i \}$.  Suppose that $\mathbb F \neq 0$ and suppose that $b < c$.  We noted above that no summand of $\mathbb F^*(-a-b)$ splits off. Using that $c \geq d \geq a+1$, we see that the smallest generator of $I_{C'}$ has degree $a+b-c  < a$, so again $C'$ allows a smaller link.  We thus have two possibilities:

\begin{enumerate}
\item $\mathbb F = 0$, in which case $I_{C'}$ has $\leq 6n$ generators, all of degree $b$.  Since $a \leq b$ and $I_{C'}$ contains elements of degree $a$, it follows that $a = b$.

\item $\mathbb F \neq 0$, in which case we need $b \geq c$ by the discussion in the preceding paragraph. But then considering the minimal free resolution (\ref{new mfr of IC}) and using that the generators of $\mathbb F$ have degrees at most $c \leq b$, in order for the generator(s) of $I_C$ having degree $b$ to participate in any syzygy, we again need $a = b$ so that the syzygies of degree $a+1$ can apply to the generator(s) of degree $b$.  But then any generator of $\mathbb F$ has degree $\leq a$ and any generator  of $\mathbb G \oplus R(-a)^2$ has degree $\leq a$. This is a contradiction to the minimality of the resolution.
\end{enumerate}

We are left with the conclusion that $\mathbb F = 0$ and $a = b$.  Hence $I_C$ has minimal free resolution
\[
0 \rightarrow R(-a-2)^n \rightarrow R(-a-1)^{4n} \rightarrow R(-a)^{3n+1} \rightarrow I_C \rightarrow 0
\]
(where the $3n+1$ comes from considering the ranks of the free modules).  But then considering the twists, we obtain $n(a+2) + (3n+1)a = 4n(a+1)$, from which it follows that $a = 2n$.  This is the minimal free resolution of the minimal curve, and so $C$ is minimal as claimed, thanks to Lemma \ref{min facts}.
\end{proof}

\begin{remark}
We believe that other liaison classes possess the property that every curve is minimally linked (in a finite number of steps) to a minimal curve, but it seems that numerical considerations of this sort will not be enough.  Indeed, we considered two natural next cases.  Both have a two-dimensional module, with one-dimensional components in each of two consecutive degrees. Specifically, we first considered
\begin{itemize}
\item $Y_1$ is a minimal Buchsbaum curve with this module (meaning that the dimensions are as given, but the structure as a graded module is trivial).  Then thanks to \cite{BM2}, we know that $\deg Y_1 = 10$ and $M(Y_1)$ occurs in degrees 2 and 3.  This curve is easy to construct with liaison addition \cite{Sw}.

\item $Y_2$ is the disjoint union of a line and a conic (which is not Buchsbaum).  This curve is minimal in its even liaison class (see Proposition \ref{facts about C1}).  The liaison properties of such curves were studied in \cite{geom-inv}.
\end{itemize}

\noindent Note that in both cases, the intermediate cohomology module is self-dual up to shift.   In both of these cases, almost the entire argument given above was able to go through.  However, using the resulting constraints as guidelines, in the end we were able to use basic double linkage starting from the minimal curve (in a careful way) to provide a {\em non-minimal} curve $C$ in the corresponding liaison class, which is minimally (directly) linked to a curve $C'$, and such that $C$ and $C'$ have the following properties:

\begin{itemize}
\item $C$ and $C'$ are cohomologically indistinguishable: they have the same degree and arithmetic genus, the same Hilbert function, and their intermediate cohomology modules occur in the same degrees.

\item $I_C$ and $I_{C'}$ have minimal free resolutions that are almost indistinguishable: the generators occur in degrees $a$ and $a+1$, but $I_C$ possesses an extra copy of $R(-a-1)$ in the first two free modules in the resolution (a so-called ghost term).

\item $C$ is minimally linked to $C'$ by a complete intersection of type $(a,a+1)$, but nevertheless, $C'$ allows a complete intersection of type $(a,a)$, linking it to a minimal curve.
\end{itemize}

\end{remark}

\begin{remark}
It should not be hard, using our methods, to find a class of codimension two subschemes of any projective space that are not minimal in their even liaison class, and yet are not minimally linked in any number of steps to a minimal element.
\end{remark}

We end this section with a natural question:

\begin{question}
Which liaison classes $\mathcal L$ of curves in $\mathbb P^3$ have the property that every curve in $\mathcal L$ is minimally linked (in a finite number of steps) to a minimal element of $\mathcal L$?  We note that ``most of the time'' there are actually two families of minimal curves, corresponding to the two even liaison classes in $\mathcal L$, but it was shown by Martin-Deschamps and Perrin \cite{MDP} that any minimal curve is minimally linked to a minimal curve in the residual class.
\end{question}



\section{Gorenstein ideals of height three}
\label{sec-Gor}

A by now classical theorem of Watanabe says that every Gorenstein
ideal of height three is licci. This is shown by induction on the
number of minimal generators using the following result:

\begin{lemma}[\cite{wat}]
  \label{lem-Watanabe}
Let $I \subset R$ be a homogeneous Gorenstein ideal of height  three
that is not a complete intersection. Let $f, g, h \in I$ be a
regular sequence such that $f, g, h$ can be extended to a minimal
generating set of $I$. Then   $I$ is linked by the complete
intersection $(f, g, h)$ of height three to an almost complete
intersection $J = (f, g, h, u)$.

Assume furthermore that $u, f, g$ is a regular sequence. Then $J$ is
linked by the complete intersection $(u, f, g)$ to a Gorenstein
ideal $I'$ whose number of minimal generators is two less than the
number of minimal generators of $I$.
\end{lemma}

\begin{proof}
This result is not stated in \cite{wat}. However, it is shown  in
the proof of \cite[Theorem]{wat} if $I$ is a Gorenstein ideal in a
regular local ring $R$. The same arguments work for a homogeneous
ideal $I$ in a polynomial ring $R$ over a field.
\end{proof}

Note that the passage from $I$ to $I'$ in this statement is an
elementary biliaison. It is a strictly descending biliaison if $\deg
u < \deg h$.

For curves, Hartshorne complemented Watanabe's result by showing the following (see also \cite{HSS} for points in $\mathbb P^3$):

\begin{theorem}[\cite{harts-Coll}]
  \label{thm-Hartsh}
Every {\em general} Gorenstein curve in $\PP^4$ can be obtained from
a complete intersection by a sequence of strictly ascending
elementary biliaisons.
\end{theorem}

The goal of this section is to strengthen both, Watanabe's  and
Hartshorne's result.

\begin{theorem}
  \label{thm-Gorenst}
Let $I \subset R$ be a homogeneous Gorenstein ideal of height
three. Then:
\begin{itemize}

\item[(a)] If $I$ is not a complete intersection then linking
$I$ minimally twice gives  a Gorenstein ideal with two fewer
generators than $I$.

\item[(b)] Consequently, $I$ is minimally licci.

\item[(c)] Furthermore, $I$ admits a sequence of strictly
decreasing CI-biliaisons down to a complete intersection.
\end{itemize}
\end{theorem}

The key to this result is to identify a particular choice for  the
element $u$ in Watanabe's lemma. Not surprisingly, its proof relies
on the structure theorem of Buchsbaum and Eisenbud \cite{BE}. The
main result in \cite{BE} says that every Gorenstein ideal of height
three is generated by the submaximal Pfaffians of an alternating map
between free modules of odd rank. This means that each such
Gorenstein ideal $I$ corresponds to a skew-symmetric matrix $M$
whose entries on the main diagonal are all zero and whose number of
rows is odd such that the $i$-th minimal generator of $I$ is the
Pfaffian of the matrix obtained from $M$ by deleting row and column
$i$ of $M$. We refer to $M$ as the {\it Buchsbaum-Eisenbud matrix} of $I$.

\begin{lemma}
  \label{lem-key}
Let $I \subset R$ be a homogeneous Gorenstein ideal of height three
with Buchsbaum-Eisenbud matrix $M$. Assume that $I$ is not a
complete intersection, and let $f, g, h \in I$ be a regular sequence
such that $f, g, h$ are the Pfaffians of the matrix obtained from
$M$ by deleting row and column $i, j$, and $k$, respectively. Then
\[
(f, g, h) : I = (f, g, h, u),
\]
where $u$ is the Pfaffian of the matrix obtained from $M$  by
deleting  rows and columns $i, j, k$.
\end{lemma}

\begin{proof}
This is essentially a consequence of \cite[Theorem 5.3]{BE}. For the
convenience of the reader we provide some details.

Let $s \geq 5$ be the number of minimal generators of $I$. Then,  by
\cite{BE}, $I$ has a minimal free resolution of the form
\[
0 \to \bigwedge^s F \to \bigwedge^{s-1} F \to F \to I \to 0,
\]
where $F$ is a free $R$-module of rank $s$ and the middle map
corresponds to an element $\ffi \in  \bigwedge^2 F$. It provides the
exterior multiplication $\lambda: \bigwedge F \to \bigwedge F,\; a
\mapsto a \wedge \ffi^{(\frac{s-3}{2})}$, where $\ffi^{(j)}$ denotes
the $j$-th divided power of $\ffi$.

Consider the minimal free resolution of the complete  intersection
$(f, g, h)$:
\[
0 \to \bigwedge^3 G \to \bigwedge^2 G \to G \to (f, g, h) \to 0.
\]
The inclusion $\iota: (f, g, h) \hookrightarrow I$ induces  a map
$\alpha: G \to F$ such that the following diagram is commutative:
\begin{equation*}
\begin{CD}
0 @>>> \bigwedge^3 G @>>> \bigwedge^2 G @>>> G @>>> (f, g, h) @>>>0 \\
& & & & & & @VV{\alpha}V  @VV{\iota}V\\
0 @>>> \bigwedge^s F  @>>> \bigwedge^{s-1} F @>>> F @>>> I @>>>0.  \\
\end{CD}
\end{equation*}
Buchsbaum and Eisenbud have determined  the remaining comparison
maps (see \cite[page 474]{BE}) such that one gets a commutative
diagram
\begin{equation*}
\begin{CD}
0 @>>> \bigwedge^3 G @>>> \bigwedge^2 G @>>> G @>>> (f, g, h) @>>>0 \\
& & @VV{\lambda_3 \circ \bigwedge^3 \alpha}V @VV{\lambda_2 \circ \bigwedge^2 \alpha}V @VV{\alpha}V  @VV{\iota}V\\
0 @>>> \bigwedge^s F  @>>> \bigwedge^{s-1} F @>>> F @>>> I @>>>0.  \\
\end{CD}
\end{equation*}
Now observe that the map $\lambda_3 \circ \bigwedge^3 \alpha$ is the
multiplication by an element, say, $u \in R$. Hence a mapping cone
argument (see \cite[Proposition 2.6]{PS}) provides that
\[
(f, g, h) : I = (f, g, h, u),
\]
It remains to identify the element $u$. By our  assumption $\{f, g,
h\}$ can be extended to a minimal generating set. Hence, using
suitable bases the map $\alpha: G \to F$ can be described by a
matrix whose columns are part of the standard basis of $F$.
Therefore the map $\bigwedge^3 \alpha: \bigwedge G \to \bigwedge^3
F$ is given by a matrix with one column whose entries are all zero
except for one, which is 1. The map $\lambda_3: \bigwedge^3 F \to
\bigwedge^s F$ is given by a matrix with one row whose entries are
the Pfaffians of order $s-3$ of $M$. It follows that the product of
these two matrices is the form $u \in R$ that equals the Pfaffian
of the  matrix obtained from $M$ by deleting each of the rows and
columns that lead to $(s-1) \times (s-1)$ matrices whose Pfaffians
are $f, g$, and $h$, respectively. This completes the argument.
\end{proof}

We are ready for the proof of the main result of this section.

\begin{proof}[Proof of Theorem \ref{thm-Gorenst}]
  We may assume that $I$ is not a complete intersection. Using induction
on the number of minimal generators we will prove all statements
simultaneously by showing that two consecutive minimal links
constitute a strictly descending elementary CI-biliaison.

Let $(f, g, h)$ be a complete intersection of height three and of
least degree inside $I$. Then $\{f, g, h\}$ can be extended to a
minimal generating set of $I$. Thus, we may apply Lemma
\ref{lem-key} and we will consider the corresponding element $u$.
Order the degrees such that $\deg f \leq \deg g \leq \deg h$. Note
that the assumption that $I$ is not a complete intersection forces $\deg
f \geq 2$. We need another estimate:
\smallskip

{\it Claim 1}: $\deg u < \deg g$.
\smallskip

Indeed, we may order the Buchsbaum-Eisenbud matrix $M$ of $I$ such
that the degrees of its entries are increasing from bottom to top
and from right to left. Then the entries on the non-main diagonal
must all have positive degree. Since $(f, g, h)$ is a complete
intersection of least degree inside $I$, the degree of $f$ is the
least degree of a minimal generator of $I$. Thus, we may assume,
using the notation of Lemma \ref{lem-key}, that $i = 1 < j < k$.
Developing the Pfaffian that gives $g$ along its left-most column,
Lemma \ref{lem-key} yields that $\deg g$ is the sum of $\deg u$ and
the degree of the $(k, 1)$ entry of $M$, which is positive. It
follows that $\deg g > \deg u$, as claimed.
\smallskip

By the choice of $f, g, h \in I$, the ideal $I$ is minimally linked
to $J = (f, g, h, u)$. The next goal is to show:
\smallskip

{\it Claim 2}: $u, f$ is an $R$-regular sequence.
\smallskip

Suppose otherwise. Then there  are forms $a, b, c \in R$ such that
$u = a b,\; f = a c$, and $\deg a \geq 1$. By symmetry of linkage,
the complete intersection $(f, g, h) = (a c, g, h)$ links $J = (a b,
a c, g, h)$ back to $I$, but one easily checks that $(a c, g, h) :
(a b, a c, g, h)$ contains $c$, contradicting our assumption that
the least degree of an element in $I$ is $\deg f$. This completes
the proof of Claim 2.
\smallskip

Now let $\fc$ be a height 3 complete intersection of least degree
inside $J$. The two claims above imply that $\fc$ is of the form
$(u, f', g')$, where  $\deg f' = \deg f$  and $\deg g'$ equals $\deg
g$ or $\deg h$. Moreover, we may assume that $f', g' \in (f, g, h)$.
In order to complete the argument we need:
\smallskip

{\it Claim 3}: $\{f', g'\}$ can be extended to a minimal generating
set of $(f, g, h)$.
\smallskip

To this end we distinguish two cases:
\begin{caselist}
  \item[Case 1:\hfill] Assume $\deg g' = \deg g$. \\
Then choose $h' \in (f, g, h)$ sufficiently general of degree $\deg
h$. Since $(f, g, h)$ has height 3, $(f', g', h')$ will also be  a
complete intersection of height 3. Moreover, it is contained in $(f,
g, h)$, and both intersections have the same degree. It follows that
$(f', g', h') = (f, g, h)$, as desired.

  \item[Case 2: \hfill] Assume $\deg g' = \deg h > \deg g$. \\
Then we may assume that there are forms $a, b \in R$ such that $g' =
a f + b g + h$. Now we choose $h' \in (f, g)$ sufficiently general
of degree $\deg g$. Then $(f', h')$ will be a complete intersection
of height 2 inside $(f, g)$. Both intersections have the same
degree, thus we get $(f', h') = (f, g)$. It follows that $(f', g',
h') = (f, af + b g +h, h')$ equals $(f, g, h)$, as wanted.

\end{caselist}
Thus, Claim 3 is established.
\smallskip

By the choice of $\fc$, the ideal $J$ is minimally linked by $\fc  =
(u, f', g')$ to an ideal $I'$. Since $I$ is minimally linked by $(f,
g, h)$ to $J$, Claim 3 combined with Watanabe's Lemma
\ref{lem-Watanabe} shows that $I'$ is a Gorenstein ideal with two
fewer minimal generators. As pointed out above, Claim 3 also shows
that the passage from $I$ to $I'$ may be viewed as an elementary
CI-biliaison. It is strictly decreasing because $\deg f' = \deg
f$,\; $\deg g \leq \deg g' \leq \deg h$, and $\deg u < \deg g$, thus
\[
\deg u + \deg f' + \deg g' < \deg f + \deg g + \deg h.
\]
This completes the proof.
\end{proof}

\begin{remark}
(i) The above result is not stated in its utmost generality at
  all. In fact, Theorem \ref{thm-Gorenst} is true whenever $I$ is a
  homogeneous Gorenstein ideal of height three in a graded
  Gorenstein algebra $R$
  over any field $K$, where the assumption includes that the
  projective dimension of $I$ over $R$ is finite.

  (ii) Parts (a) and (b) of Theorem \ref{thm-Gorenst} are also true
  for every
  Gorenstein ideal of height three in a local Gorenstein ring $R$.

  (iii) If we were only interested in parts (b) and (c),
  Lemma \ref{lem-key} (or even Lemma \ref{lem-Watanabe}) would not be
necessary, and the main ingredients of a much shorter proof are
already contained in the above proof.  However, in proving (a) we
have the additional benefit of getting at the heart of the structure
of Gorenstein ideals of codimension three with respect to liaison.

(iv) There are many instances in the literature illustrating the principle that codimension two arithmetically Cohen-Macauay ideals and codimension three arithmetically Gorenstein ideals have properties that closely parallel one another.  (See \cite{migbook} Section 4.3 for an extensive list with references.)  Theorem \ref{thm-Gorenst} is another illustration --  part (a) is also true for codimension two Cohen-Macaulay ideals, and indeed it is Gaeta's approach to proving his theorem.
\end{remark}

We illustrate our result by giving various examples.

\begin{example}
Let $Z$ be a set of 8 points on a twisted cubic curve in $\mathbb
P^3$.   Then $Z$ is arithmetically Gorenstein, with $h$-vector
$(1,3,3,1)$.  The smallest link for $Z$ is of type $(2,2,3)$, so the
residual, $Z'$, has $h$-vector $(1,2,1)$. One can check that the
smallest link for $Z'$ is of type $(1,2,3)$ (as a codimension three
subscheme of $\mathbb P^3$), so we have that $1 = \deg u < \deg f =
2$.
\end{example}


\begin{example}\label{BD ex1}
Let $Z$ be an arithmetically Gorenstein subscheme of $\mathbb P^3$
constructed as follows.  Choose a point $P$ in $\mathbb P^3$ and let
$\Lambda$ be a plane through $P$ and let $\lambda$ be a line through
$P$ not in $\Lambda$.  Let $X$ be a complete intersection of type
$(3,3)$ in $\Lambda$ containing $P$, and let $Z_1 = X \backslash P$.
Let $Z_2$ be a set of four points on $\lambda$, none equal to $P$.
Then  $Z = Z_1 \cup Z_2$ is arithmetically Gorenstein \cite{BD} with
minimal free resolution
\[
0 \rightarrow R(-7) \rightarrow
\begin{array}{c}
R(-3) \\
\oplus \\
R(-4)^2 \\
\oplus \\
R(-5)^2
\end{array}
\rightarrow
\begin{array}{c}
R(-2)^2 \\
\oplus \\
R(-3)^2 \\
\oplus \\
R(-4)
\end{array}
\rightarrow I_Z \rightarrow 0
\]
and $h$-vector $(1,3,4,3,1)$.
One easily checks that $Z$ has a minimal link of type $(2,3,4)$, linking $Z$ to an almost complete intersection $Z'$ with minimal free resolution
\[
0 \rightarrow
\begin{array}{c}
R(-6) \\
\oplus \\
R(-7)
\end{array}
\rightarrow
\begin{array}{c}
R(-4)^2 \\
\oplus \\
R(-5)^2 \\
\oplus \\
R(-6)
\end{array}
\rightarrow
\begin{array}{c}
R(-2)^2 \\
\oplus \\
R(-3) \\
\oplus \\
R(-4)
\end{array}
\rightarrow I_{Z'} \rightarrow 0.
\]
We first note that $Z'$ also has $h$-vector $(1,3,4,3,1)$,  rather
than having smaller degree than $Z$.  However, by Claim 2 above,
$Z'$ has a regular sequence of type $(2,2)$.  We saw that $Z'$ is
not a complete intersection. Since $Z'$ has degree 12, the minimal
link for $Z'$ must be of type $(2,2,4)$.  One checks that $Z'$ is
then linked to a complete intersection of type $(1,2,2)$. In this
example we have $\deg u = 2 = \deg f$.
\end{example}

\begin{example} \label{BD ex2}
In Example \ref{BD ex1}, suppose $X$ is a complete intersection of type $(6,6)$ on $\Lambda$ and $Z_2$ is a set of 10 points on $\lambda$, then $Z = Z_1 \cup Z_2$ is again arithmetically Gorenstein, this time  with minimal free resolution
\[
0 \rightarrow R(-13) \rightarrow
\begin{array}{c}
R(-3) \\
\oplus \\
R(-7)^2 \\
\oplus \\
R(-11)^2
\end{array}
\rightarrow
\begin{array}{c}
R(-2)^2 \\
\oplus \\
R(-6)^2 \\
\oplus \\
R(-10)
\end{array}
\rightarrow I_Z \rightarrow 0.
\]
A minimal link is of type $(2,6,10)$, with residual $Z'$ having minimal free resolution
\[
0 \rightarrow
\begin{array}{c}
R(-12) \\
\oplus \\
R(-16)
\end{array}
\rightarrow
\begin{array}{c}
R(-7)^2 \\
\oplus \\
R(-11)^2 \\
\oplus \\
R(-15)
\end{array}
\rightarrow
\begin{array}{c}
R(-2) \\
\oplus \\
R(-5) \\
\oplus \\
R(-6) \\
\oplus \\
R(-10)
\end{array}
\rightarrow I_{Z'} \rightarrow 0.
\]
In this case  $\deg u = 5 > 2 = \deg f$.
\end{example}

\begin{example}
  \label{ex-counter}
For codimension two arithmetically Cohen-Macaulay subschemes, Gaeta in fact proved the stronger result that if one always links using minimal generators, not necessarily minimal links, it still holds that the ideal obtained in the second step has two fewer minimal generators than the original ideal (in fact, one link provides an ideal with one fewer minimal generator, but this does not have an analog here).  Hence even weakening the minimal link condition to simply links by minimal generators, we can still conclude that the ideal is licci.  We now show that this does not extend to the Gorenstein situation.

Let $I$ be a Gorenstein ideal of height 3 with a minimal free
resolution of the form:
\[
0 \to R(-7) \to R(-4)^7 \to R(-3)^7 \to I \to 0.
\]
Linking $I$ minimally by a complete intersection of type $(3, 3,
3)$, we get an ideal $J$ with  minimal free resolution:
\[
0 \to R(-6)^4 \to R(-5)^7 \to
\begin{array}{c}
R(_-3)^3 \\
\oplus \\
R(-2)
\end{array}
\to J \to 0.
\]
Now we choose three sufficiently general cubic forms in $J$. They
generate a complete intersection $\fc$ and are minimal generators of
$J$. Linking $J$ by $\fc$ we get a Gorenstein ideal $I'$ whose
resolution has the same shape as the one of $I$. In particular, $I$
and $I'$ have the same number of minimal generators. Notice that the
second link to get $I'$ is not minimal, thus this example does not
contradict Theorem \ref{thm-Gorenst}(a). However,  this example
shows that to draw the conclusion of Theorem \ref{thm-Gorenst}(a),
it is not enough to assume only that the two links both use minimal
generators of the ``starting'' ideal.
\end{example}

\begin{remark} As mentioned in the introduction, Theorem \ref{thm-Gorenst} generalizes a result of Hartshorne \cite{harts-Coll} and a result of Hartshorne, Sabadini and Schlesinger \cite{HSS}, in addition to sharpening Watanabe's result.
\end{remark}



\begin{thebibliography}{ll}

\bibitem{BBM} E.\ Ballico, G.\ Bolondi and J.\ Migliore, {\em The Lazarsfeld-Rao
Problem for Liaison Classes of Two-Codimensional Subschemes of} $\mathbb P^{n}$,
Amer.\ J.\  Math.\ {\bf 113} (1991), 117--128.

\bibitem{BD} C.\ Bocci and  G.\ Dalzotto,
{\em Gorenstein points in $\mathbb P^3$}, in:
Liaison and related topics (Turin, 2001).  Rend.\ Sem.\ Mat.\ Univ.\ Politec.\ Torino {\bf 59} (2001),  no.\ 2, 155--164 (but appeared in 2003).

\bibitem{BM1} G.\ Bolondi and J.\ Migliore, {\em Classification of Maximal Rank
Curves in the Liaison Class} $L_n$, Math.\ Ann.\ {\bf 277} (1987), 585--603.

\bibitem{BM2} G.\ Bolondi and J.\ Migliore, {\em Buchsbaum Liaison
Classes}, J.\ Algebra {\bf 123} (1989), 426--456.

\bibitem{BM4} G.\ Bolondi and J.\ Migliore, {\em The Structure of an even
liaison  class}, Trans.\ Amer.\ Math.\ Soc.\ {\bf 316} (1989), 1--37.


\bibitem{BM6} G.\ Bolondi and J.\ Migliore, {\em The Lazarsfeld-Rao property on
an arithmetically Gorenstein variety}, Manuscripta Math.\ {\bf 78} (1993), 347--368.

\bibitem{BE}
D.\ Buchsbaum and D.\ Eisenbud, {\em Algebra structures for finite
free resolutions, and some structure theorems for ideals of
codimension 3}, Amer.\ J.\ Math.\ {\bf 99} (1977), 447--485.

\bibitem{CoCoA} CoCoA: {\em A system for Doing Computations in Commutative Algebra.}  Available at {\tt http://cocoa.dima.unige.it.}

\bibitem{denegri-valla} E.~De Negri and G.~Valla, {\em The $h$-vector of a
Gorenstein codimension three domain}, Nagoya Math.\ J.\ {\bf 138} (1995),
113--140.

\bibitem{GM4} A.V.\ Geramita and J.\ Migliore, {\em A Generalized Liaison
Addition}, J.\ Algebra {\bf 163} (1994), 139--164.

\bibitem{gaeta-aCM} F.\ Gaeta, {\em Sulle curve sghembe algebriche di residuale
finito}, Annali di Matematica s.\ IV, t.\ XXVII (1948), 177--241.

\bibitem{GM5}  A.V.\ Geramita and J.\ Migliore, {\em Reduced Gorenstein
Codimension Three Subschemes of Projective Space}, Proc.\ Amer.\ Math.\ Soc.\
{\bf 125} (1997), 943-950.

\bibitem{GLM} S.\ Guarrera, A.\ Logar and E.\ Mezzetti, {\em An algorithm for computing minimal curves},
Arch.\ Math.\ (Basel) {\bf 68} (1997), no. 4, 285--296.

\bibitem{harts-Coll}
R.\ Hartshorne, {\em Geometry of arithmetically Gorenstein curves in
$\PP^4$}, Coll.\ Math.\ {\bf 55} (2004), 97--111.

\bibitem{hartshorne-toulouse} R.\ Hartshorne, {\em On Rao's theorems and the Lazarsfeld-Rao property}, Ann. Fac.\ Sciences Toulouse XII, no.\ 3 (2003), 375--393.

\bibitem{HMN}
R.\ Hartshorne,  J.\ Migliore, and U.\ Nagel, {\em Liaison addition and the structure of a
Gorenstein liaison class}, J.\ Algebra {\bf 319} (2008), 3324--3342.

\bibitem{HSS} R.\ Hartshorne, I.\ Sabadini, and E.\ Schlesinger,
{\em Codimension 3 arithmetically Gorenstein subschemes of
projective  $N$-space}, Ann.\ Inst. Fourier (to appear).

\bibitem{HTV} J.\ Herzog, N.V.\ Trung and G.\ Valla, {\em On hyperplane
sections of reduced irreducible varieties of low codimension}, J.\ Math.\ Kyoto
Univ.\ {\bf 34-1} (1994), 47--72.

\bibitem{HUduke}
C. Huneke and B. Ulrich. \emph{Algebraic linkage}, Duke Math.\ J.\
\textbf{56} (1988), 415--429.

\bibitem{HMNU} C.\ Huneke, J.\ Migliore, U.\ Nagel, and B.\ Ulrich, {\em Minimal homogeneous liaison and licci ideals}, in: ``Algebra, Geometry and their Interactions (Notre Dame 2005),"  Contemp.\ Math.\ {\bf  448} (2007), 129--139.

\bibitem{KMMNP} J.\ Kleppe, J.\ Migliore, R.M.\ Mir\'o-Roig, U.\ Nagel, and C.\
Peterson, {\em Gorenstein liaison, lomplete lntersection liaison invariants and
unobstructedness}, Mem.\ Amer.\ Math.\ Soc.\ {\bf 154} (2001), no.\
  732, 116 pp.

\bibitem{MDP} M.\ Martin-Deschamps and D.\ Perrin, ``Sur la Classification
des Courbes Gauches,'' Ast\'erisque {\bf 184--185}, Soc.\ Math.\ de France, 1990.

\bibitem{J-thesis} J.\ Migliore, Topics in the Theory of Liaison of Space Curves, Ph.D.\ thesis, Brown University, 1983.

\bibitem{geom-inv} J.\ Migliore, {\em Geometric Invariants for Liaison of Space
Curves}, J.\ Algebra {\bf 99} (1986), 548--572.

\bibitem{migbook} J.\ Migliore, ``Introduction to Liaison Theory and
Deficiency Modules,''   Progress in Mathematics {\bf 165}, Birkh\"auser, 1998.

\bibitem{London} J.\ Migliore, {\em Submodules of the deficiency module,} J.\ London Math.\
Soc.\ {\bf 48}(3) (1993), 396--414.

\bibitem{MN6} J.\ Migliore and U.\ Nagel, {\em Liaison and related topics: Notes
from the Torino Workshop/School}, Rend.\ Sem.\ Mat.\ Univ.\
     Politec.\ Torino {\bf 59} (2001), no.\ 2, 59--126
(but appeared in 2003).

\bibitem{uwepaper} U.\ Nagel, {\em Even liaison classes generated by
Gorenstein linkage}, J.\ Algebra {\bf 209} (1998), no. 2, 543--584.

\bibitem{nollet} S.\ Nollet, {\em Even Linkage Classes},  Trans.\ Amer.\
Math.\ Soc.\ {\bf 348} (1996), no. 3, 1137--1162.

\bibitem{PS}
C.\ Peskine and L.\ Szpiro,  {\em Liaison des vari\'et\'es
alg\'ebriques.\ I}, Invent.\ Math.\ {\bf 26} (1974), 271--302.

\bibitem{rao1} P.\ Rao, {\em Liaison among curves in} $\mathbb P^{3}$, Invent.\ Math.\
{\bf 50} (1979), 205--217.

\bibitem{rao2} P.\ Rao, {\em Liaison equivalence classes}, Math.\ Ann.\ {\bf 258}
(1981), 169--173.

\bibitem{Sw} P.\ Schwartau, {\em Liaison addition and monomial ideals}, Ph.D.\
thesis, Brandeis University, 1982.

\bibitem{wat}
J.\ Watanabe, {\em A note on Gorenstein rings of embedding
codimension 3}, Nagoya Math.\ J.\ {\bf 50} (1973), 227--232.


\end{thebibliography}
\end{document}